\documentclass[10pt]{amsart}

\usepackage[leqno]{amsmath}
\usepackage{amsthm}
\usepackage{amsfonts}
\usepackage{amssymb}
\usepackage{eucal}
\usepackage{graphicx}
\DeclareGraphicsExtensions{.pdf,.png,.jpg}
\usepackage{hyperref}
\usepackage[all]{xy}
\CompileMatrices

\newcommand{\setof}[1]{\ensuremath{\left \{ #1 \right \}}}

\DeclareFontFamily{OT1}{rsfs}{}
\DeclareFontShape{OT1}{rsfs}{n}{it}{<-> rsfs10}{}
\DeclareMathAlphabet{\mathscr}{OT1}{rsfs}{n}{it}

\theoremstyle{plain}
  \newtheorem{theorem}{Theorem}
  \newtheorem{conjecture}{Conjecture}
  
  \newtheorem{lemma}{Lemma}
  
  \newtheorem*{rh}{Riemann hypothesis}

\theoremstyle{definition}

\theoremstyle{remark}

\numberwithin{equation}{section}

\title{On the Riemann Hypothesis and Hilbert's Tenth Problem}
\author{Aran Nayebi \\ Written: February 2012}
\email{anayebi@stanford.edu}
\urladdr{http://www.stanford.edu/~anayebi}
\begin{document}
\begin{abstract}
The Riemann hypothesis is one of the most famous unresolved problems in modern mathematics. 
The discussion here will present an overview of past methods that prove the Riemann hypothesis is a $\Pi_1^0$ sentence.
We also end with some attempts towards showing the Elliott-Halberstam conjecture is $\Pi_1^0$.
\end{abstract}

\maketitle
\tableofcontents

\section{Introduction}
The Riemann zeta function is the function of the complex variable $s$ defined in the half-plane $\Re(s) > 1$ by the sum of the infinite series
\begin{equation*}
\zeta(s) := \sum_{n = 1}^{\infty}\frac{1}{n^s}.
\end{equation*}
By analytic continuation, we can define the zeta function on the entire complex plane $\mathbb{C}$ except at $s = 1$ where $\zeta(s)$ has its unique pole. In his seminal 1859 memoir entitled ``On the number of primes less than a given magnitude'', Riemann \cite{riemann} derives an analytic formula for the distribution of prime numbers expressed in terms of the zeros of $\zeta(s)$. The function $\zeta(s)$ has as \emph{trivial zeros} the negative even integers -2, -4, $\ldots$ and as \emph{nontrivial zeros} the complex numbers $\frac{1}{2} + i\beta$, where $\beta$ is a zero of the even entire function $\xi(t)$ of the complex variable $t$ defined (by Riemann) as
\begin{equation*}
\xi(t) := \frac{1}{2}s(s-1)\pi^{-s/2}\Gamma\left(\frac{s}{2}\right)\zeta(s),
\end{equation*}
where $s = \frac{1}{2} + it$ and, as usual, $\Gamma(s)$ is the gamma function. The zeros of $\xi(t)$ have imaginary part between $-i/2$ and $i/2$. The Riemann hypothesis is the statement that all zeros of $\xi(t)$ are real, or equivalently,
\begin{rh}
The nontrivial zeros of $\zeta(s)$ have real part of $\frac{1}{2}$.
\end{rh}
\begin{figure}
\includegraphics[scale=0.7]{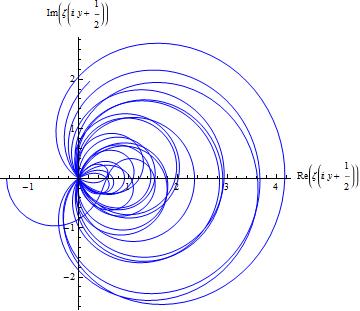}
\caption{Plot of the real and imaginary parts of $\zeta(s)$ along the critical line $\Re(s) = \frac{1}{2}$ as $y$ is varied varied from 0 to 80.}
\label{figure1}
\end{figure}
From a first glance, it may seem that the Riemann hypothesis is only an interesting property of $\zeta(s)$, and one may be led to think that even Riemann himself took that view, for he wrote ``Without doubt it would be desirable to have a rigorous proof of this proposition; however I have left this research aside for the time being after some quick unsuccessful attempts, because it appears to be unnecessary for the immediate goal of my study'' \cite[pp. 3-4]{bombieri2}. However, one should not conclude that the Riemann hypothesis was of only minor interest to him, since from his memoir it is quite likely that he was aware of the importance of his hypothesis in approximating the distribution of prime numbers, as we will explain below. \newline
\indent The Riemann hypothesis has many startling consequences for analytic number theory.  If we let $\pi(x)$ denote the number of primes $\le x$ which is well approximated by
\begin{equation*}
\operatorname{Li}(x) = \int_2^x \frac{dt}{\log t},
\end{equation*}
then the Riemann hypothesis is equivalent to saying that
\begin{equation*}
\pi(x) = \operatorname{Li}(x) + O(x^{1/2}\log x).
\end{equation*}
The error term is the best possible in this case since it is known to oscillate in both directions to an order of at least $\text{Li}\left(x^{1/2}\right)\log^{(3)} x$, where $\log^{(3)} x$ is the three-fold iterated logarithm.\footnote{Von Koch \cite{koch} in fact proved that the Riemann hypothesis is the best possible bound for the error of the prime number theorem (which states that $\pi(x) \sim \frac{x}{\log x}$).} A more precise result due to Schoenfeld \cite{schoenfeld} is that the Riemann hypothesis is equivalent to
\begin{equation*}
|\pi(x) - \operatorname{Li}(x)| < \frac{1}{8\pi}x^{1/2}\log x,
\end{equation*}
for all $x \ge 2657$. The Riemann hypothesis is one of the main open questions in prime number theory due to the fact that its failure would wreak havoc in the distribution of prime numbers \cite{bombieri2}. Moreover, it has become a central problem of pure mathematics as well, for instance due to its consequences for global $L$-functions. A discussion of these functions is beyond the scope of this article; however, the reader is referred to \S 2 of Bombieri's \cite{bombieri2} for an informal yet informative treatment of the topic. \newline
\indent Our interest in the Riemann hypothesis stems from the consequences of the negative solution to Hilbert's tenth problem by Davis, Putnam, and Robinson \cite{dpr1961}, and subsequently by Matiyasevich \cite{mat1970}. In 1900, listed as problem 10 of his famous 23 problems\footnote{It is of relevance to point out that the Riemann hypothesis is part of Hilbert's eighth problem.}, the eminent mathematician Hilbert asked,
\begin{center}
``Given a diophantine equation with any number of unknown quantities and with rational integral numerical coefficients: \emph{To devise a process according to which it can be determined by a finite number of operations whether the equation is solvable in rational integers}'' \cite[pg. 458]{hilbert}.
\end{center}
The term `rational integers' here simply refers to the integers, and the words `process' and `finite number of operations' refers to the modern notion of an algorithm.\footnote{The formalization of the concept of the algorithm began with the 1931 incompleteness theorems of G\"{o}del \cite{godel1931} regarding the unsolvability of the \emph{Entscheidungsproblem} (`decision problem'). Subsequent formalizations were centered upon defining `effective calculability', which included the recursive functions of Herbrand \cite{herbrand1931} and G\"{o}del \cite{godel1934}, Church's $\lambda$-calculus \cite{church}, Post's `Formulation 1' \cite{post}, and Turing's \emph{a}-machines \cite{turing}.} In clearer terms, Hilbert was essentially asking for a general algorithm to decide whether a given diophantine equation with integer coefficients has a solution in integers. The negative solution to this problem is commonly referred to as the DPRM theorem\footnote{This result is either referred to as the MRDP (Matiyasevich-Robinson-Davis-Putnam) theorem or the DPRM (Davis-Putnam-Robinson-Matiyasevich) theorem (Matiyasevich favors the latter and Davis the former \cite{jackson}). To preserve the historical integrity of the order of succession of results that led to the negative solution to Hilbert's tenth problem, we will use the latter naming convention in this article.}, and formally states that
\begin{theorem}[DPRM theorem]
Every recursively enumerable set is diophantine.\footnote{Equivalently, by the DPRM theorem, it follows that every $\Pi_1^0$ proposition is equivalent to the unsolvability of a particular diophantine equation.}
\end{theorem}
Since there exist recursively enumerable sets which are not computable, then it is an immediate consequence that Hilbert's tenth problem is unsolvable over $\mathbb{Z}$. Therefore, a Turing machine cannot decide whether a given diophantine equation has solutions or not over $\mathbb{Z}$.\footnote{The analog of Hilbert's tenth problem over $\mathbb{Q}$ has still not been proven, although Poonen \cite{poonen} poses a conjecture, that if true, would imply the unsolvability of Hilbert's tenth problem over $\mathbb{Q}$. Furthermore, the unsolvability of Hilbert's tenth problem has been proven by several authors for many rings. A recent paper by Mazur and Rubin \cite{mazur-rubin} shows that conditional on the 2-primary part of the Shafarevich-Tate conjecture, Hilbert's tenth problem is unsolvable over the ring of integers of \emph{every} number field. Shlapentokh \cite{s-2011} provides an excellent survey of the developments arising from the attempts to solve the analog of Hilbert's tenth problem for the field of rational numbers and the rings of integers of number fields.} \newline
\indent Ultimately, the DPRM theorem bridges computability theory and number theory (this relationship, for instance, is made incredibly explicit by the nondeterministic diophantine machine (NDDM) of Adleman and Manders \cite{adleman-manders}),  and since recursively enumerable sets are diophantine, it is well-known many open problems in number theory can be expressed as the assertion that a particular diophantine equation has no solution in nonnegative integers. The Riemann hypothesis is one of these problems, and there have been several reductions of the Riemann hypothesis to a system of diophantine equations (and therefore to a single polynomial that is unsolvable in nonnegative integers if the Riemann hypothesis is true, and solvable if otherwise). The first formulation \cite{dmr1976} is based on a decidability condition that considers the behavior of a Cauchy integral on a path around the zeros of the Riemann zeta function and involves the computable function $\delta(x)$ defined for positive integers as such
\begin{equation*}
\delta(x) := \prod_{n < x} \prod_{j \le n} \eta(j),
\end{equation*}
where $\eta(j) = 1$ unless $j$ is a prime power, and $\eta(p^k) = p$. The second, much simpler formulation \cite{mat-book} utilizes the relation
\begin{equation*}
\operatorname{explog}(a,b) \Leftrightarrow \exists x \setof{x > b + 1\mbox{           }\&\mbox{           }\left(1 + \frac{1}{x}\right)^{xb} \le a + 1 < 4 \left(1 + \frac{1}{x}\right)^{xb}}.
\end{equation*}
Essentially, it is possible to write a computer program to search for a counterexample to the Riemann hypothesis, and this program will halt with a counterexample if and only if a certain diophantine equation has a solution.\footnote{Of course, the diophantine representations of these famous problems are not currently considered a \emph{viable} approach towards a proof.} \newline
\indent This article will proceed as follows. In \S 2 we will discuss in detail both approaches to the reduction of the Riemann hypothesis to a system of diophantine equations. And in \S 3 we will discuss two significant open problems in number theory that have not yet been proven to be $\Pi_1^0$ propositions.

\section{Proof that the Riemann hypothesis is a $\Pi_1^0$ proposition}
The key step in proving that the Riemann hypothesis is a $\Pi_1^0$ proposition is to prove a decidable property (this is a nontrivial task) which can then be written in terms of system of diophantine equations, from which all that remains is to combine these equations into a single equation (which is easily done either by the simple sum of squares method or, to minimize the number of variables, by the more specialized Matiyasevich-Robinson relation-combining theorem \cite{mat-rob1975}). The reason why the decidable property is essential is that if we have an algorithm that tells us in a finite number of steps whether an arbitrary natural number has a particular property $P$, then we can reduce the assertion that every natural number has the property $P$ to the unsolvability of a particular diophantine equation which we can explicitly write down. Hence, when discussing the two approaches to proving that the Riemann hypothesis is $\Pi_1^0$, we will focus on the derivation of the decidable property, rather than the explicit system of equations that is produced. And in the case of the Riemann hypothesis, it is not immediately evident that it can be formulated in the desired form $\forall n$ $P(n)$.
\subsection{Davis, Matiyasevich, and Robinson (1976)}
As mentioned before, Davis, Matiyasevich, and Robinson's \cite{dmr1976} approach to establishing a decidable property involves the computable function $\delta(x)$ defined for positive integers as such
\begin{equation*}
\delta(x) := \prod_{n < x} \prod_{j \le n} \eta(j),
\end{equation*}
where $\eta(j) = 1$ unless $j$ is a prime power, and $\eta(p^k) = p$. From which they proved the following decidable property
\begin{theorem}[Davis, Matiyasevich, and Robinson]\label{dp1}
The Riemann hypothesis is equivalent to the assertion $($for $n = 1, 2, 3, \ldots)$ that
\begin{equation} \label{r1}
\left(\sum_{k \le \delta(n)}\frac{1}{k} - \frac{n^2}{2}\right)^2 < 36n^3.
\end{equation} 
\end{theorem}
It is rather simple, using the methods of Matiyasevich \cite{mat1970}, to obtain an explicit diophantine equation which has no solution in case the decidable property above holds. Furthermore, the authors remark that the constant 36 can be readily improved. Indeed, after we present the proof of Theorem ~\ref{dp1}, it will be a trivial observation to the reader that one can obtain the slightly improved bound
\begin{equation} \label{1}
\left(\sum_{k \le \delta(n)}\frac{1}{k} - \frac{n^2}{2}\right)^2 < (\gamma^2 + 8\gamma + 16)n^3 + (2\gamma + 8)n^{3/2} + 1,
\end{equation}
where $\gamma = 0.577\ldots$ is the Euler-Mascheroni constant. Note that the system of diophantine equations cannot contain $\gamma$ because it cannot be written explicitly. A bound to Theorem~\ref{dp1} which does not involve $\gamma$ but obviously weaker than the one presented in \eqref{1} is
\begin{equation} \label{2}
\left(\sum_{k \le \delta(n)}\frac{1}{k} - \frac{n^2}{2}\right)^2 < 25n^3 + 10n^{3/2} + 1.
\end{equation}
Of course, neither \eqref{1} or \eqref{2} are intended to be \emph{practical} replacements to the bound given in Theorem ~\ref{dp1}; we simply presented them to affirm that it is an easy task, as Davis, Matiyasevich, and Robinson pointed out, to improve the bound in \eqref{r1}. \newline
\indent Setting child's play aside, we proceed to present the proof of Theorem ~\ref{dp1}. First, the following basic facts (from number theory) must be stated, which will be used in the proof, and therefore we will present these as lemmas without proof (the references that contain their proofs can be found on page 335 of \cite{dmr1976} in the same order as they are listed here).
\begin{lemma} \label{l1}
It is well-known that
\begin{equation*}
\frac{\Gamma^{'}(3/2)}{\Gamma(3/2)} =  2 - \gamma - \log 4.
\end{equation*}
\end{lemma}
\begin{lemma} \label{l2}
If we let $R$ to denote the set of nontrivial zeros of $\zeta(s)$,
\begin{equation*}
\frac{\zeta^{'}(s)}{\zeta(s)} = \log 2\pi - 1 - \frac{\gamma}{2} - \frac{1}{s-1} - \frac{1}{2}\frac{\Gamma^{'}\left(\frac{s}{2} + 1\right)}{\Gamma\left(\frac{s}{2} + 1\right)} + \sum_{\rho \in R} \left(\frac{1}{s-\rho} + \frac{1}{\rho}\right).
\end{equation*}
Moreover,
\begin{equation*}
\lim_{s\to 1}\left(\frac{\zeta^{'}(s)}{\zeta(s)} + \frac{1}{s-1}\right) = \gamma.
\end{equation*}
\end{lemma}
\begin{lemma} \label{l3}
For $x \in \mathbb{R}$, where $x > 0$,
\begin{equation*}
\psi(x) = \sum_{p^k \le x}\log p = \sum_{j \le x} \log \eta(j),
\end{equation*}
and
\begin{equation*}
\psi_1(x) = \int_1^x \psi(u)du.
\end{equation*}
Thus, for integral $x$, $\psi(x) \le \sum_{j \le x}\log x = x\log x \le x^{3/2}$ and $\psi_1(x) = \log \delta(x)$. This gives us the following expression for $\psi_1(x)$ for $x \ge 1$,
\begin{equation*}
\psi_1(x) = \frac{x^2}{2} - \sum_{\rho\in R}\frac{x^{\rho+1}}{\rho(\rho+1)} - x\frac{\zeta^{'}(0)}{\zeta(0)} + \frac{\zeta^{'}(-1)}{\zeta(-1)} - \sum_{r = 1}^{\infty}\frac{x^{1-2r}}{2r(r-1)}.
\end{equation*}
\end{lemma}
\begin{lemma} \label{l4}
We also have that
\begin{equation*}
\frac{\zeta^{'}(0)}{\zeta(0)} = \log 2\pi,
\end{equation*}
and
\begin{equation*}
\frac{\zeta^{'}(1)}{\zeta(1)} = 12\log A - 1 = 1.985\ldots,
\end{equation*}
where $A$ is the Glaisher-Kinlekin constant defined exactly as such
\begin{equation*}
A = \lim_{n\to\infty}\frac{H(n)}{n^{n^2/2 + n/2 + 1/12}e^{-n^2/4}},
\end{equation*}
where $H(n)$ is the hyperfactorial function.
\end{lemma}
\begin{lemma} \label{l5}
\begin{equation*}
\left(\sum_{k \le n}\frac{1}{k}\right) - 1 < \log n < \sum_{k \le n}\frac{1}{k},
\end{equation*}
for $n$ a positive integer.
\end{lemma}
We need one additional lemma, whose proof will be shown, as it is derived from the above lemmas.
\begin{lemma} \label{l6}
\begin{equation*}
\sum_{\rho \in R}\frac{1}{\rho(1-\rho)} = \gamma + 2 - \log 4\pi.
\end{equation*}
\end{lemma}
\begin{proof}
We have that
\begin{equation*}
\sum_{\rho \in R} \frac{1}{\rho(1-\rho)} = \sum_{\rho \in R}\left(\frac{1}{1-\rho} + \frac{1}{\rho}\right).
\end{equation*}
From Lemma ~\ref{l2} it follows that
\begin{equation*}
\sum_{\rho \in R}\left(\frac{1}{1-\rho} + \frac{1}{\rho}\right) = \lim_{s\to 1}\left(\frac{\zeta^{'}(s)}{\zeta(s)} + \frac{1}{s-1}\right) - \log 2\pi + 1 + \frac{\gamma}{2} + \frac{1}{2}\frac{\Gamma^{'}(3/2)}{\Gamma(3/2)}.
\end{equation*}
Lastly, from Lemmas ~\ref{l1} and ~\ref{l2}, we have then that
\begin{equation*}
\lim_{s\to 1}\left(\frac{\zeta^{'}(s)}{\zeta(s)} + \frac{1}{s-1}\right) - \log 2\pi + 1 + \frac{\gamma}{2} + \frac{1}{2}\frac{\Gamma^{'}(3/2)}{\Gamma(3/2)} = \gamma + 2 - \log 4\pi.
\end{equation*}
\end{proof}
Now, assume the Riemann hypothesis, namely that $\rho \in R$ implies $\Re(\rho) = \frac{1}{2}$. Moreover, for $\rho \in R$, the complex conjugate $\bar{\rho} = 1 - \rho$. Thus,
\begin{equation*}
\sum_{\rho \in R}\frac{1}{|\rho| \cdot |\rho + 1|} \le \sum_{\rho \in R}\frac{1}{|\rho|^2} = \sum_{\rho \in R} \frac{1}{\rho\bar{\rho}} = \sum_{\rho \in R}\frac{1}{\rho(1-\rho)}.
\end{equation*}
Hence, from Lemma ~\ref{l6},
\begin{equation}\label{bound}
\sum_{\rho \in R}\frac{1}{|\rho|\cdot |\rho + 1|} \le \gamma + 2 - \log 4\pi.
\end{equation}
Using \eqref{bound} and Lemma ~\ref{l3}, and noting that the Riemann hypothesis implies that for $\rho \in R$, $x \ge 1$, we have that $|x^{\rho + 1}| = x^{3/2}$, then
\begin{equation*}
|\psi_1(x) - \frac{x^2}{2}| \le x^{3/2}\left(\sum_{\rho \in R}\frac{1}{|\rho|\cdot |\rho+1|} + |\frac{\zeta^{'}(0)}{\zeta(0)}| +  |\frac{\zeta^{'}(-1)}{\zeta(-1)}| + \sum_{r = 1}^{\infty}\frac{1}{2r(2r-1)}\right).
\end{equation*}
A subsequent simplification of the infinite sum gives us,
\begin{equation*}
\sum_{r = 1}^{\infty}\frac{1}{2r(2r-1)} = \sum_{r = 1}^{\infty}\left(\frac{1}{2r-1} - \frac{1}{2r}\right) = \sum_{n = 1}^{\infty}\frac{(-1)^{n+1}}{n} = \log 2.
\end{equation*}
And with the additional use of Lemma ~\ref{l4},
\begin{equation*}
|\psi_1(x) - \frac{x^2}{2}| \le x^{3/2}\left(\gamma + 2 - \log 4\pi + \log 2\pi + 2 + \log 2\right) = (4+\gamma)x^{3/2} < 5x^{3/2}.
\end{equation*}
But by Lemma ~\ref{l5},
\begin{equation*}
|\sum_{k \le \delta(x)}\frac{1}{k} - \psi_1(x)| < 1.
\end{equation*}
As a result, we have
\begin{equation*}
|\sum_{k \le \delta(x)}\frac{1}{k} - \frac{x^2}{2}| < 1 + 5x^{3/2} \le 6x^{3/2}.
\end{equation*}
Thus,
\begin{equation*}
\left(\sum_{k \le \delta(x)}\frac{1}{k} - \frac{x^2}{2}\right)^2 < 36x^3,
\end{equation*}
which gives us \eqref{r1}. \newline
\indent To prove the implication in the other direction, suppose that \eqref{r1} holds for all positive integers. Then by Lemma ~\ref{l5}, for $x = 1, 2, 3, \ldots$ ,
\begin{equation*}
|\psi_1(x) - \frac{x^2}{2}| < 1 + 6x^{3/2}.
\end{equation*}
Hence, for any real $x \ge 1$, by Lemma ~\ref{l3},
\begin{equation*}
|\psi_1(x) - \frac{x^2}{2}| \le \int_{[x]}^{x}\psi(u)du + |\psi_1([x])-\frac{[x]^2}{2}| + |\frac{[x]^2 - x^2}{2}| < \psi([x]) + 1 + 6x^{3/2} + x \le 9x^{3/2}.
\end{equation*}
We now employ arguments that are modeled after those used by Ingham \cite{ingham}. For $s > 1$, it is a well-known fact
\begin{equation*}
\zeta(s) = \prod_{p}\frac{1}{1-p^{-s}},
\end{equation*}
where $p$ runs through all primes. Therefore,
\begin{equation*}
\log \zeta(s) = - \sum_{p}\log(1-p^{-s}) = \sum_{p,m}\frac{1}{mp^{ms}},
\end{equation*}
where $m$ runs through all positive integers. By differentiation,
\begin{equation*}
-\frac{\zeta^{'}(s)}{\zeta(s)} = \sum_{p}\frac{p^{-s}\log p}{1-p^{-s}} = \sum_{p,m}\frac{\log p}{p^{ms}}.
\end{equation*}
In terms of the von Mangoldt function, $\Lambda(n)$, we can rewrite this as,
\begin{equation*}
-\frac{\zeta^{'}(s)}{\zeta(s)} = \sum_{n = 1}^{\infty}\frac{\Lambda(n)}{n^s},
\end{equation*}
where $\Lambda(n)$ is $\log p$ if $n$ is a positive power of a prime $p$, and is 0 otherwise. Moreover, note that $\psi(x) = \sum_{n \le x}\Lambda(n)$. Therefore, we have that
\begin{equation*}
-\frac{\zeta^{'}(s)}{\zeta(s)} = s\int_1^{\infty}\frac{\psi(x)}{x^{s+1}}dx.
\end{equation*}
Or, in terms of $\psi_1(x)$,
\begin{equation*}
-\frac{\zeta^{'}(s)}{\zeta(s)} = s(s+1)\int_1^{\infty}\frac{\psi_1(x)}{x^{s+2}}dx.
\end{equation*}
Hence,
\begin{equation} \label{inequality}
-\frac{\zeta^{'}(s)}{\zeta(s)} - \frac{s(s+1)}{2(s-1)} = s(s+1)\int_1^{\infty}\frac{\psi_1(x)-\frac{x^2}{2}}{x^{s+2}}dx.
\end{equation}
So, assuming our inequality,
\begin{equation*}
|\frac{\psi_1(x) - \frac{x^2}{2}}{x^{s+2}}| \le \frac{9}{|x^{s+1/2}|},
\end{equation*}
the integral in \eqref{inequality} uniformly converges on the region $\Re(s) > \sigma_0$, for any $\sigma_0 > 1/2$. \eqref{inequality} must be valid for $\Re(s) > 1/2$, and the function on the left is analytic in this domain. But clearly, this implies that $\zeta(s)$ does not vanish for $1/2 < \Re(s) < 1$, which means that the Riemann hypothesis holds. \newline
\indent Theorem ~\ref{dp1} is now proved.
\subsection{Matiyasevich (1993)}
Of course, Theorem ~\ref{dp1} is not the only method of proving that the Riemann hypothesis is a $\Pi_1^0$ statement. Another, albeit lesser known, formulation is due to Matiyasevich \cite{mat-book} and involves the usage of less variables in its diophantine representation than the Davis, Matiyasevich, and Robinson \cite{dmr1976} result. Matiyasevich essentially takes advantage of the fact that the Riemann hypothesis implies that for
\begin{equation} \label{m1}
n \ge 600,
\end{equation}
\begin{equation} \label{m2}
|\psi(n) - n| < n^{1/2}\log^2 n.
\end{equation}
Since
\begin{equation*}
\psi(n) = \log(\operatorname{lcm}(1,2,\ldots,n)),
\end{equation*}
we can reduce \eqref{m2} to a system of diophantine equations because the function lcm (least common multiple) is diophantine. However, this is only immediately obvious when lcm is considered as a function of two arguments, not with indefinite numbers of arguments. Namely, the diophantine representation of lcm for positive integers $b$ and $c$ is
\begin{equation*}
a = \operatorname{lcm}(b,c) \Leftrightarrow bc = a\operatorname{gcd}(b,c),
\end{equation*}
where the generalized diophantine representation for gcd is
\begin{equation*}
a = \operatorname{gcd}(b,c) \Leftrightarrow bc > 0\mbox{           }\&\mbox{           }a \mid b\mbox{           }\&\mbox{           }a \mid c\mbox{           }\&\mbox{           }\exists xy \{a = bx-cy\},
\end{equation*}
and as usual $\mid$ denotes the relation of divisibility which is diophantine since
\begin{equation*}
a \mid b \Leftrightarrow \exists x \{ax = b\}.
\end{equation*}
Matiyasevich \cite{mat-book} avoids this problem by way of a bounded universal quantifier, namely, the number $m$ is a common multiple of $1,2,\ldots,n$ if and only if
\begin{equation} \label{m3}
\forall y < n\{(y+1)\mid m\}
\end{equation}
and is the least common multiple if, in addition,
\begin{equation} \label{m4}
m > 0\mbox{           }\&\mbox{           }\forall y < m \{y = 0 \lor \exists x < n\{(x+1)\nmid y\}\}.
\end{equation}
Now, the next problem we face is whether or not $\log n$ is diophantine. Since its values are not generally integers, one method would be to show that $\lfloor \log n \rfloor$ is diophantine, but this is a nontrivial task. Instead, Matiyasevich \cite{mat-book} introduces the relation mentioned before, namely,
\begin{equation} \label{explog}
\operatorname{explog}(a,b) \Leftrightarrow \exists x \setof{x > b + 1\mbox{           }\&\mbox{           }\left(1 + \frac{1}{x}\right)^{xb} \le a + 1 < 4 \left(1 + \frac{1}{x}\right)^{xb}}.
\end{equation}
We will first prove that for all $a$ and $b$,
\begin{equation} \label{i1}
\operatorname{explog}(a,b) \Rightarrow |b-\log(a+1)| < 2,
\end{equation}
and second, that
\begin{equation} \label{i2}
\forall a \exists b \{\operatorname{explog}(a,b)\}.
\end{equation}
We proceed as follows. It is a well-known fact from calculus that for $x > 0$,
\begin{equation*}
\left(1 + \frac{1}{x}\right)^{x} < e < \left(1 + \frac{1}{x}\right)^{x+1}.
\end{equation*}
Raising this inequality to the power $b$, we have
\begin{equation*}
\left(1 + \frac{1}{x}\right)^{xb} < e^b < \left(1 + \frac{1}{x}\right)^{(x+1)b}.
\end{equation*}
Dividing by $e$, we have,
\begin{equation*}
\frac{e^b}{e} \le \frac{(1+\frac{1}{x})^{(x+1)b}}{(1+\frac{1}{x})^x} < \left(1+\frac{1}{x}\right)^{xb} \le a + 1 < 4\left(1+\frac{1}{x}\right)^{xb} \le 4e^b,
\end{equation*}
which implies \eqref{i1}. \newline
\indent Now, for any $a$ one can find $b$ such that
\begin{equation*}
b \le \log(a+1) < b + \log 3.
\end{equation*}
Choosing a very large $x$, we can make both $\left(1+\frac{1}{x}\right)^x$ and $\left(1+\frac{1}{x}\right)^{x+1}$ arbitrarily close to $e$ since
\begin{equation*}
e = \lim_{x\to\infty}\left(1+\frac{1}{x}\right)^{x} = \lim_{x\to\infty}\left(1+\frac{1}{x}\right)^{x+1},
\end{equation*}
thereby validating the inequalities in \eqref{explog}.
\begin{theorem}[Matiyasevich] \label{mat-red}
The negation of the Riemann hypothesis is equivalent to the existence of numbers $k$, $l$, $m$, and $n$ satisfying conditions \eqref{m1}, \eqref{m2}, \eqref{m3}, and
\begin{equation} \label{m4}
\operatorname{explog}(m-1,l),
\end{equation}
\begin{equation} \label{m5}
\operatorname{explog}(n-1,k),
\end{equation}
\begin{equation} \label{m6}
(l-n)^2 > 4n^2 k^4.
\end{equation}
\end{theorem}
\begin{proof}
First, let us assume that such $k$, $l$, $m$, and $n$ exist. Then by \eqref{m2} and \eqref{m3},
\begin{equation} \label{s1}
m = e^{\psi(n)}.
\end{equation}
Moreover, by \eqref{i1}, it follows from \eqref{m5} that
\begin{equation} \label{s2}
|l - \psi(n)| < 2.
\end{equation}
And from \eqref{m6} that
\begin{equation} \label{s3}
|k - \psi(n)| < 2.
\end{equation}
As a result, we obtain
\begin{equation*}
|\psi(n) - n| \ge |l-n| - |l-\psi(n)| > 2n^{1/2}k^2 - 2 > n^{1/2}\log^2 n,
\end{equation*}
which with \eqref{m2}, gives us the desired contradiction. \newline
\indent On the other hand, assume that the Riemann hypothesis is not true. Since the Riemann hypothesis is equivalent to the statement
\begin{equation*}
\psi(n) = n + O(n^{1/2}\log^2 n),
\end{equation*}
then there is a number $n$ satisfying \eqref{m1} such that
\begin{equation*}
|\psi(n)-n| > 10 n^{1/2}\log^2 n.
\end{equation*}
Choosing $m$ according to \eqref{s1}, conditions \eqref{m3} and \eqref{m4} will be satisfied. Then by \eqref{i2}, we can find numbers $k$ and $l$ satisfying \eqref{m4} and \eqref{m5}. By \eqref{i1}, the inequalities \eqref{s2} and \eqref{s3} will be satisfied, which in turn imply that
\begin{equation*}
|l-n| \ge |\psi(n)-n|-|l-\psi(n)| > 10n^{1/2}\log^2 n - 2 > 2n^{1/2}(\log n+2)^2 > 2n^{1/2}k^2,
\end{equation*}
thereby satisfying the inequality \eqref{m6}.
\end{proof}
Since all of the conditions given in Theorem ~\ref{mat-red} are diophantine, then we can construct a diophantine equation whose unsolvability is equivalent to the Riemann hypothesis.
\section{Open problems that have not yet been shown to be $\Pi_1^0$ sentences}
Now, the natural question is whether or not every problem can be reduced to a question concerning the solvability of diophantine equations. Davis, Matiyasevich, and Robinson \cite{dmr1976} are keen to observe that not all open problems in number theory have been proven to be $\Pi_1^0$ sentences. The example they present is the twin prime conjecture, namely, the statement
\begin{equation*}
\forall n \{\exists p \{p>n\mbox{        }\&\mbox{           }\text{$p$ prime}\mbox{        }\&\mbox{           }\text{$p+2$ prime}\}\}.
\end{equation*} 
As Matiyasevich \cite{mat-book} mentions, since we can eliminate only bounded universal quantifiers, we have no obvious way to restate the twin prime conjecture as the problem of the solvability or unsolvability of a particular diophantine equation. However, as Davis, Matiyasevich, and Robinson \cite{dmr1976} suggest, if we strengthen the conjecture to the statement,
\begin{equation*}
\forall n \{\exists p \{n+4<p<2^{n+4}\mbox{        }\&\mbox{           }\text{$p$ prime}\mbox{        }\&\mbox{           }\text{$p+2$ prime}\}\},
\end{equation*}
then this strong twin prime conjecture can be transformed into the unsolvability of a particular diophantine equation. Note that like the twin prime conjecture, the strong twin prime conjecture is also widely believed to be true.
\newline
\indent We will add to this list another number-theoretic problem that has not been shown to be a $\Pi_1^0$ proposition. This is the Elliott-Halberstam conjecture, which is related to the Riemann hypothesis in that, if true, it results in deeper estimates for the distribution of primes in arithmetic progressions than the generalized Riemann hypothesis (which is a generalization of the Riemann hypothesis to a statement regarding the distribution of the zeros of global $L$-functions) does. We shall now quantify what we mean by `deeper'. Given an arithmetic progression $a \pmod q$, we let $\pi(x; q, a)$ denote the number of primes $\le x$ lying in the progression with $\gcd(a,q) = 1$.
\begin{equation*}
\pi(x; q, a) = \frac{1}{\phi(q)}{\operatorname{Li}(x)} + E(x; q, a),
\end{equation*}
where $\phi(q)$ is Euler's totient function and $E(x; q, a)$ is the error term. For every $A > 0$ there is a constant $\theta(q,A)$ that may depend on $q$ for which
\begin{equation*}
|E(x; q, a)| \le \theta(q,A)\frac{x}{(\log x)^{A}}.
\end{equation*}
This result is rather weak; however, under the assumption of the generalized Riemann hypothesis, we derive the much stronger bound
\begin{equation} \label{GRH}
|E(x; q, a)| \le \theta x^{1/2} \log x,
\end{equation}
for a positive constant $\theta$ \emph{independent} of $q$ if $q \le x$. \newline
\indent In 1965, Vinogradov \cite{vinogradov} and Bombieri \cite{bombieri} independently improved Barban's results \cite{barban} regarding Linnik's large sieve, and proved the following
\begin{theorem}[Bombieri-Vinogradov theorem]\label{BV}
For any positive constant $A$ there exists a positive constant $B = B(A)$ such that $Q = {x^{1/2}}{(\log x)^{-B}}$, with
\begin{equation} \label{bv-eq}
\sum_{q \le Q} E{^{*}}(x;q) \ll \frac{x}{(\log x)^{A}},
\end{equation}
where for a modulus $q$, the functions $E(x;q)$ and $E{^{*}}(x;q)$ are defined as
\begin{equation*}
\begin{split}
E(x;q) &=  \max_{\substack{\gcd(a,q)=1}}|E(x;q,a)| \\
E{^{*}}(x;q) &= \max_{\substack{y \le x}}E(y,q).
\end{split}
\end{equation*}
\end{theorem}
$B$ can be computed explicitly, a possible value of $B$ is $3A+23$ \cite{bombieri}; $B = 24A+46$ is also permissible \cite{sound}. From the Bombieri-Vinogradov theorem, we have on average over $q \le Q$,
\begin{equation} \label{bound1}
E(x;q) \le \theta x^{1/2}(\log x)^{B-A}.
\end{equation}
The bound given in \eqref{bound1} can sometimes be used as a substitute for the bound \eqref{GRH} that assumes the generalized Riemann hypothesis. Thus, for many situations that involve sufficiently many Dirichlet $L$-functions, and therefore require the use of the generalized Riemann hypothesis, the Bombieri-Vinogradov theorem can be employed instead. \newline
\indent Bombieri, Friedlander, and Iwaniec \cite{bometal1} \cite{bometal2} have managed to extend Theorem ~\ref{BV} beyond the $x^{1/2}$ barrier. Specifically, they proved that $Q$ can be taken more explicitly as
\begin{equation*}
Q = x^{1/2} \exp\Big\{\frac{\log x}{(\log^{(2)} x)^{B}}\Bigr\}.
\end{equation*}
\indent We say that primes have a level of distribution $\vartheta$ if \eqref{bv-eq} holds for any $A>0$ and any $\epsilon >0$ with
\begin{equation*} 
Q=x^{\vartheta -\epsilon}.
\end{equation*}
According to the Bombieri-Vinogradov theorem, primes have a level of distribution $1/2$. Elliott and Halberstam \cite{EH} extended Theorem ~\ref{BV} to conjecture that the primes have a level of distribution 1, namely,
\begin{conjecture}[Elliott-Halberstam conjecture] \label{EH}
For any positive constant $A$ and any $\epsilon > 0$ there exists a positive constant $B = B(A)$ such that $Q = x^{1-\epsilon}$, with
\begin{equation*}
\sum_{q \le Q} E{^{*}}(x;q) \ll \frac{x}{(\log x)^{A}}.
\end{equation*}
\end{conjecture}
Friedlander, Granville, Hildebrand, and Maier \cite{fghm} have shown that for $Q$ a bit larger, namely if
\begin{equation*}
Q = x \cdot \exp\Big\{{-\frac{1}{4}}(A-1)(\log^{(2)} x)^2 (\log^{(3)} x)^{-1}\Bigr\},
\end{equation*}
for $A > 1$, then Conjecture ~\ref{EH} does not hold. In fact, it is widely believed that the Elliott-Halberstam conjecture is valid with $Q$ any fixed power $x^{\mu}$, where $0 < \mu < 1$ \cite{elliott}. \newline
\indent The Elliott-Halberstam conjecture has many startling consequences. Alford, Granville, and Pomerance \cite{agp} proved that assuming the conjecture, then for any constant $a < 1$ there exists $C > 0$ such that there are more than $x^{a}$ Carmichael numbers up to $x$, once $x > C$. Another, more recent startling consequence was proven by Goldston, Pintz, and  Y{\i}ld{\i}r{\i}m \cite{gpy} in their work on the twin prime conjecture, who showed that assuming Conjecture ~\ref{EH}, for infinitely many $n$,
\begin{equation} \label{touse}
p_{n+1} - p_{n} \le 16,
\end{equation}
where $p_{n}$ denotes the $n$-th prime. \newline
\indent We note that Conjecture ~\ref{EH} lies deeper than the generalized Riemann hypothesis, and Bombieri \cite{bombieri} remarks that nothing more precise than the Bombieri-Vinogradov theorem of Theorem ~\ref{BV} can be obtained even under the assumption of the generalized Riemann hypothesis. Hence, it is of little surprise that not much progress has been made towards proving the Elliott-Halberstam conjecture. \newline
\indent Moreover, as the reader may recall, we mentioned that determining a decidable property is not immediately evident in the case of the Riemann hypothesis. This is even more true with the Elliott-Halberstam conjecture, and for good reason. The author of this article had the idea to use \eqref{touse} of Goldston, Pintz, and Y{\i}ldr{\i}m \cite{gpy} to formulate a decidable property. To explain in more detail, \eqref{touse} can be expressed in terms of the solutions of the Pell equation
\begin{equation} \label{pell}
x^2 - (a^2 - 1)y^2 = 1,
\end{equation}
for $a > 0$. This can be shown to be possible by first letting the solutions of \eqref{pell} in order of the size of $y$ be denoted as
\begin{equation*}
x = \chi_{a}(n), \quad y = \psi_{a}(n),
\end{equation*}
and by utilizing the result from Theorem 419 of Hardy and Wright \cite{intro}
\begin{equation*}
\lfloor \theta_1 \cdot 10^{2^{n}} \rfloor = 10^{2^{n}} \sum_{k=1}^{n}p_k 10^{{-2}^{k}},
\end{equation*}
\begin{equation*}
\lfloor \theta_1 \cdot 10^{2^{n-1}} \rfloor = 10^{2^{n-1}} \sum_{k=1}^{n-1}p_k 10^{{-2}^{k}},
\end{equation*}
and
\begin{equation*}
\lfloor \theta_1 \cdot 10^{2^{n+1}} \rfloor = 10^{2^{n+1}} \sum_{k=1}^{n+1}p_k 10^{{-2}^{k}}.
\end{equation*}
Giving us,
\begin{equation} \label{p-n}
p_n = \lfloor \theta_1 \cdot 10^{2^{n}} \rfloor - 10^{2^{n-1}} \lfloor \theta_1 \cdot 10^{2^{n-1}} \rfloor,
\end{equation}
and
\begin{equation} \label{p-n+1}
p_{n+1} = \lfloor \theta_1 \cdot 10^{2^{n+1}} \rfloor -  10^{2^{n}} \lfloor \theta_1 \cdot 10^{2^{n}} \rfloor.
\end{equation}
Although \eqref{p-n} is not practical for strictly evaluating $p_{n}$ (since the constant $\theta_1 \approx 0.0203000500000007$ should be evaluated to $2^{n}$ decimal places, which can only be done if the values of $p_1,\ldots, p_n$ are known), its form is useful in expressing the $n$-th prime in terms of $\psi_{a}(n)$, which Matiyasevich and Robinson \cite{mat-rob1975} prove is diophantine. Unfortunately, apart from the issue of including the real number $\theta_1$ to formulate a decidable property involving \eqref{touse}, \eqref{touse} only is \emph{implied by} but is not equivalent to the Elliott-Halberstam conjecture. It turns out, however, that the task of finding a statement equivalent to the Elliott-Halberstam conjecture (let alone proving that this statement can be translated into the solvability or unsolvability of a particular diophantine equation) is incredibly nontrivial. Not to mention given that the translation of an open problem to a particular diophantine equation does not provide much insight into solving the problem, such efforts into finding a decidable property for the Elliott-Halberstam conjecture can appropriately be viewed as fruitless. However, this certainly does not preclude the possibility that a breakthrough in the study of diophantine equations may provide insight into showing that a class of equations (including perhaps an equation of interest in that class that is equivalent to the truth of an important open problem) is unsolvable. One can only anticipate such a result.
\newpage
\bibliographystyle{amsplain}

\end{document}